\documentclass[11pt]{article}
\usepackage[T1]{fontenc}
\usepackage[utf8]{inputenc}
\usepackage{setspace}
\usepackage{newtxtext,newtxmath}
\usepackage[margin=1.25in]{geometry}
\usepackage{url}
\usepackage{array}
\usepackage{graphicx}
\usepackage{subcaption}
\usepackage{color}
\makeatletter
\@ifundefined{openbox}{}{}
\makeatother

\usepackage{amsmath,amssymb,amsthm,mathtools}
\usepackage{bm}
\usepackage{booktabs}
\usepackage{tabularx}
\usepackage{enumitem}
\usepackage[dvipsnames]{xcolor}
\usepackage[colorlinks=true,allcolors=green!60!black]{hyperref}

\usepackage{bbm}
\usepackage{multirow}
\usepackage{authblk}
\usepackage[round,compress]{natbib}
\usepackage{algorithm}
\usepackage{algorithmic}

\renewcommand{\det}[1]{\left| #1 \right|}

\renewcommand{\P}{{\bf P}}

\newcommand{\R}{{\mathbb R}}
\renewcommand{\S}{{\bf S}}

\newcommand{\ben}{\begin{enumerate}}
\newcommand{\een}{\end{enumerate}}
\newcommand{\beq}{\begin{equation}}
\newcommand{\eeq}{\end{equation}}

\usepackage{booktabs,array}

\newcount\rowc

\newcommand{\diff}{\mathrm{d}}
\newcommand{\vnorm}[1]{\left\lVert #1 \right\rVert}
\newcommand{\Pbb}{\mathbb{P}}
\newcommand{\Ebb}{\mathbb{E}}

\def\P{\mathbb{P}}
\setlist*[enumerate]{label=(\roman*)}

\newtheorem{theorem}{Theorem}[section]
\newtheorem{lemma}[theorem]{Lemma}

\newtheorem{remark}[theorem]{Remark}

\newtheorem{assumption}{Assumption}
\newtheorem{definition}{Definition}

\title{A New Look at Bayesian Testing}
\author[1]{Jyotishka Datta}
\author[2]{Nicholas G. Polson}
\author[3]{Vadim Sokolov}
\author[4]{Daniel Zantedeschi}
\affil[1]{Department of Statistics, Virginia Tech}
\affil[2]{Booth School of Business, University of Chicago}
\affil[3]{Department of Systems Engineering and Operations Research, George Mason University}
\affil[4]{Muma College of Business, University of South Florida; \texttt{danielz@usf.edu}}
\date{}

\begin{document}

\maketitle
\begin{abstract}
We identify the critical deviation scale governing Bayesian evidence accumulation in regular parametric testing. Under integrated Bayes risk with zero--one loss, the risk-optimal rejection boundary lies in a moderate deviation regime, with a square-root logarithmic inflation relative to the usual local asymptotic normal scale. Under Cram\'er regularity, local prior smoothness at the null, and symmetric loss, we derive the sharp threshold and show that its leading logarithmic term is universal across regular priors, while lower-order constants depend on the local prior density, Fisher information, and prior model odds. The result extends to one-parameter exponential families through local asymptotic normality and places Jeffreys' testing threshold, the Bayesian information criterion penalty, and Chernoff--Stein type error-exponent arguments within a common asymptotic moderate deviation framework.
\end{abstract}


\section{Introduction}\label{sec:intro}

In regular parametric models, classical testing operates at the local asymptotic normal scale: deviations of order $n^{-1/2}$ from the null yield statistics of order $O(1)$, and the critical value $z_{1-\alpha/2}$ is held fixed as $n$ grows. Bayesian evidence accumulation requires a larger threshold. The reason is structural: the Bayes factor compares the null density $f(\bm{x}\mid\theta_0)$ with the marginal likelihood $m(\bm{x})=\int f(\bm{x}\mid\theta)\,\pi(\diff\theta)$, which integrates over the alternative parameter space. This integration introduces a complexity penalty, because prior mass allocated to alternatives near $\theta_0$ dilutes the evidence, and a stronger departure from the null is needed before the data decisively favour $H_a$. The central question is: \emph{at what deviation scale does this evidence become decisive?}

The answer involves a logarithmic correction to the classical LAN scale. When the rejection boundary is chosen to minimize integrated Bayes risk, the critical deviation scale is
\begin{equation}\label{eq:scales}
\lambda_n \asymp \sqrt{\frac{\log n}{n}}\,,\qquad
t_n = \frac{\sqrt{n}\,\lambda_n}{\sigma} \asymp \sqrt{\log n}\,,\qquad
t_n^2 = \log n + O(1).
\end{equation}
These three equivalent expressions, in the parameter, standardized statistic, and squared-statistic scales, characterize the boundary at which the type~I and type~II risk contributions are asymptotically balanced. The leading term $\log n$ is universal across regular models and priors; the $O(1)$ constants depend on the local prior density at $\theta_0$, the Fisher information, and the prior model odds.

This moderate deviation boundary clarifies several phenomena that have appeared separately in the testing literature: Jeffreys' $\sqrt{\log n}$ threshold for Bayes factor calibration, Lindley's paradox arising from the divergence between this boundary and the fixed-$\alpha$ critical value, the BIC model selection penalty $(d/2)\log n$, the Rubin--Sethuraman moderate deviation calculus for Bayes risk, and the polynomial prefactors in Chernoff--Stein error exponents. All arise from the same asymptotic regime in which Gaussian tail probabilities, evaluated at $\sqrt{\log n}$-scale deviations, decay polynomially rather than exponentially. The present paper places these phenomena on a common asymptotic coordinate system.

Beyond unification, the paper makes three technical contributions. First, we derive the critical deviation scale \eqref{eq:scales} under explicit regularity conditions, namely the Cram\'er condition, prior smoothness, and symmetric loss, and obtain the sharp threshold formula with identified constants and remainder $O\!\bigl(\sqrt{(\log n)/n}\bigr)$. Second, we prove a uniform moderate deviation lemma providing the tail approximations that drive the analysis, and show via Laplace expansion of the marginal likelihood that the $\sqrt{\log n}$ scaling is universal across regular priors. Third, we extend the threshold to one-parameter exponential families under local asymptotic normality, and provide algebraic comparisons with fixed-$\alpha$ and e-value calibrations.

\section{Testing framework and Bayes risk}\label{sec:setup}

\subsection{Model and notation}\label{sec:model}

Let $X_1,\ldots,X_n$ be i.i.d.\ from $\{P_\theta:\theta\in\Theta\subset\R\}$ with densities $f(x\mid\theta)$ and write $\bm{x}=(x_1,\ldots,x_n)$. We test $H_0:\theta=\theta_0$ versus $H_a:\theta\neq\theta_0$ using a decision rule $\delta:\mathcal{X}^n\to\{0,1\}$ with $\delta(\bm{x})=1$ denoting rejection. Under $H_a$, $\theta$ has prior density $\pi(\cdot)$ on $\Theta$, and the prior model probabilities are $\pi_0=\Pr(H_0)$ and $\pi_a=\Pr(H_a)$.

The parameter deviation from the null is $\lambda_n=|\hat\theta-\theta_0|$, and the standardized statistic is $t_n=\sqrt{n}\,\lambda_n/\sigma$, where $\sigma^2=\operatorname{Var}_{\theta_0}(X_1)$.

\subsection{Deviation regimes}\label{sec:regimes}

The scale of $\lambda_n$ determines three asymptotic regimes, each producing qualitatively different tail behaviour.

\begin{definition}[Deviation scales]\label{def:regimes}
\begin{enumerate}
\item \textbf{CLT regime:} $\lambda_n=O(n^{-1/2})$, equivalently $t_n=O(1)$.
Tail probabilities $\Pbb(|\bar{X}_n-\theta_0|>\lambda_n)=O(1)$.

\item \textbf{Moderate deviation regime:} $\lambda_n\to 0$ and $n\lambda_n^2\to\infty$.
On the Rubin--Sethuraman boundary $\lambda_n=a\sqrt{(\log n)/n}$ with $a>0$ fixed,
$t_n=(a/\sigma)\sqrt{\log n}$,
and tail probabilities decay polynomially:
\[
\Pbb(|\bar{X}_n-\theta_0|>\lambda_n) \sim C\,n^{-a^2/(2\sigma^2)}/\sqrt{\log n}.
\]

\item \textbf{Large deviation regime:} $\lambda_n=c$ fixed.
Tail probabilities decay exponentially:
$\Pbb(|\bar{X}_n-\theta_0|>c) = \exp\{-nI(c)+o(n)\}$,
where $I(c)=\sup_t\{tc-\psi(t)\}$ is Cram\'er's rate function and $\psi(t)=\log\Ebb[e^{tX_1}]$.
\end{enumerate}
\end{definition}

The moderate deviation regime interpolates between the CLT and large deviation scales. In the CLT regime, deviations of order $n^{-1/2}$ produce $O(1)$ tail probabilities, so the tail does not decay and the rejection rate is constant in $n$. In the large deviation regime, deviations of order $O(1)$ produce exponentially decaying probabilities $\exp(-nI)$, and the likelihood ratio test has exponentially vanishing error. Between these extremes, the moderate deviation scale $\lambda_n=a\sqrt{(\log n)/n}$ yields the polynomial tail decay $n^{-a^2/(2\sigma^2)}$. This is precisely the decay rate at which Bayesian risk balancing occurs, and the constant $a$ can be chosen to equalize the type~I and type~II risk contributions.

Fixed-$\alpha$ classical testing operates on the CLT scale with constant critical value $z_{1-\alpha/2}$. Simple-versus-simple Bayes risk is governed by the large deviation scale via the Chernoff--Stein exponent. The Bayes-rule threshold for composite-alternative testing lies on the intermediate RS boundary, where both likelihood concentration and prior mass allocation influence the boundary.

\subsection{Bayes risk decomposition}\label{sec:risk-decomp}

Denote the null density by $f(\bm{x}\mid\theta_0)$ and the Bayes mixture by
$m(\bm{x})=\int_\Theta f(\bm{x}\mid\theta)\,\pi(\diff\theta)$.
The log Bayes factor in favour of $H_0$ is
$\log BF_{01}(\bm{x}) = \log f(\bm{x}\mid\theta_0) - \log m(\bm{x})$.

The identity $m(\bm{x})\,p(\theta\mid\bm{x}) = f(\bm{x}\mid\theta)\,\pi(\theta)$
yields two forms of the integrated Bayes risk under $0$--$1$ loss.

\textit{Posterior-risk form} (integrate $\theta$ first, then $\bm{x}$):
\[
B_n(\pi,\delta)=\int m(\bm{x})\left\{\int L\bigl(\theta,\delta(\bm{x})\bigr)\,p(\theta\mid\bm{x})\,\diff\theta\right\}\diff\bm{x}.
\]
This yields the pointwise Bayes rule: $\delta^\star(\bm{x})=\mathbf{1}\{\P(\theta\in\Theta_0\mid\bm{x})<\tfrac{1}{2}\}$,
equivalently $BF_{01}<\pi_a/\pi_0$ \citep{berger2003could}.

\textit{Prior-risk form} (integrate $\bm{x}$ first, then $\theta$):
\[
B_n(\pi,\delta)=\pi_0\,\alpha_n+\pi_a\,\beta_n,
\]
where $\alpha_n$ and $\beta_n$ are the type~I and integrated type~II error probabilities.
This form exposes how threshold rules $\delta(\bm{x})=\mathbf{1}\{T_n(\bm{x})>c_n\}$ can be calibrated by moderate deviations of $T_n$ \citep{rubin1965probabilities,RubinSethuraman1965_BRE}.

\subsection{Interpretation of the Bayes-rule boundary}\label{sec:risk-interp}

The Bayes-rule rejection boundary arises from balancing the two terms in $B_n=\pi_0\alpha_n+\pi_a\beta_n$. The type~I component $\alpha_n$ is controlled by the tail probability of the test statistic under $H_0$: for a threshold rule rejecting when $|t_n|>c$, this is $\Pbb_{\theta_0}(|t_n|>c)$, which decays exponentially in~$c^2$ by Gaussian tail asymptotics. The type~II component $\beta_n$ integrates the probability of failing to reject over the alternative prior, and involves the marginal likelihood $m(\bm{x})$. For alternatives near the null, the posterior concentrates at rate $\sqrt{n}$, producing a polynomial factor $n^{-1/2}$ in the integrated type~II error.

Intuitively, the $\exp(-c^2/2)$ decay of the type~I tail must match the $n^{-1/2}$ from posterior concentration to minimize total risk. The moderate deviation scale emerges because these two decay rates, exponential in $c^2$ for type~I and polynomial in $n$ for type~II, balance precisely when $c^2\sim\log n$. At this threshold, the Gaussian tail $\exp(-c^2/2)$ produces a polynomial factor matching the $n^{-1/2}$ from the alternative integration. If the threshold were held fixed (as in classical testing), type~I error would not decay, while type~II error would vanish; the resulting procedure would be asymptotically suboptimal under integrated Bayes risk. Conversely, if the threshold grew too quickly, type~I error would be negligible but type~II error would dominate. The logarithmic rate $c\sim\sqrt{\log n}$ is the unique balancing point.

\subsection{BIC and the Laplace connection}\label{sec:bic}

The BIC approximation to the log marginal likelihood,
\[
\log m(\bm{x})=\log f(\bm{x}\mid\hat\theta)-\frac{d}{2}\log n+O(1),
\]
arises from a Laplace expansion about the MLE under standard regularity.
Both BIC and the moderate deviation boundary share the same algebraic mechanism.
In model selection, the Laplace prefactor contributes the additive penalty $-(d/2)\log n$;
at the RS boundary $\lambda_n=a\sqrt{(\log n)/n}$,
\[
\Pbb\!\bigl(|\bar{X}_n-\theta_0|>\lambda_n\bigr)
\approx \exp\!\left\{-\frac{a^2}{2\sigma^2}\log n\right\}=n^{-a^2/(2\sigma^2)},
\]
producing a polynomial factor through the same logarithmic correction on the log-scale.

\section{Moderate deviation approximations}\label{sec:md-approx}

Throughout, $\Pbb_\theta$ and $\Ebb_\theta$ refer to the sampling distribution $P_\theta^n$.

\begin{assumption}[Cram\'er condition]\label{ass:cramer}
The cumulant generating function $\psi(t)=\log\Ebb[e^{tX_1}]$ exists in a neighbourhood of the origin.
\end{assumption}

\begin{assumption}[Prior regularity]\label{ass:prior}
The prior density $\pi$ under $H_a$ satisfies $\pi\in C^3(\R)$ with $\pi(\theta_0)\in(0,\infty)$, and $\pi$ has at most polynomial growth: $|\pi(\theta)|\le C(1+|\theta|^m)$ for some $C,m>0$.
\end{assumption}

\begin{lemma}[Uniform moderate deviation approximation]\label{lem:uniform-md}
Under Assumption~\ref{ass:cramer}, let $\sigma^2=\psi''(0)=\operatorname{Var}_{\theta_0}(X_1)\in(0,\infty)$ and $\lambda_n=a\sqrt{(\log n)/n}$ with $a>0$ fixed. Then
\begin{align*}
\Pbb_{\theta_0}\!\bigl(|\bar{X}_n-\theta_0|>\lambda_n\bigr)
  &= \frac{2\sigma}{\sqrt{2\pi n}\,\lambda_n}\,
     \exp\!\Bigl({-}\frac{n\lambda_n^2}{2\sigma^2}\Bigr)\,(1+o(1))\\
  &= \frac{\sqrt{2}\,\sigma}{a\sqrt{\pi\log n}}\,
     n^{-a^2/(2\sigma^2)}\,(1+o(1)),
\end{align*}
as $n\to\infty$, uniformly over $a$ in compact subsets of $(0,\infty)$.
\end{lemma}

\begin{proof}
Set $z_n:=\sqrt{n}\,\lambda_n/\sigma=(a/\sigma)\sqrt{\log n}$.
By the Cram\'er moderate deviation theorem \citep{rubin1965probabilities}, under
Assumption~\ref{ass:cramer}, for $z_n\to\infty$ with $z_n=o(n^{1/6})$,
\[
\Pbb_{\theta_0}\!\bigl(\bar{X}_n-\theta_0>\lambda_n\bigr)
=\bigl(1-\Phi(z_n)\bigr)\bigl(1+O(z_n^3/\sqrt{n})\bigr).
\]
Since $z_n=(a/\sigma)\sqrt{\log n}=o(n^{1/6})$ and $z_n^3/\sqrt{n}=O\bigl((\log n)^{3/2}/\sqrt{n}\bigr)\to 0$,
the correction vanishes. Mills' ratio $(1-\Phi(z))\sim\varphi(z)/z$ as $z\to\infty$ then gives
\[
\Pbb_{\theta_0}\!\bigl(\bar{X}_n-\theta_0>\lambda_n\bigr)
\sim \frac{1}{\sqrt{2\pi}}\frac{e^{-z_n^2/2}}{z_n}
=\frac{\sigma}{a\sqrt{2\pi\log n}}\,n^{-a^2/(2\sigma^2)}.
\]
Doubling for the two-sided event and substituting $\lambda_n=a\sqrt{(\log n)/n}$ gives both displays.
Uniformity over compact subsets of $a\in(0,\infty)$ follows from the uniformity of the Cram\'er approximation in $z_n$; see \citet{eichelsbacher2002moderate}.
\end{proof}

\begin{remark}[Saddlepoint refinement]\label{rem:saddlepoint}
Saddlepoint methods \citep{daniels1954saddlepoint} provide the density of $\hat\theta$ with relative error $O(n^{-1})$:
\begin{equation}
f(\hat\theta\mid\theta)\approx\left(\frac{n}{2\pi\,\psi''(t^*)}\right)^{1/2}
\exp\Big(n\big[\psi(t^*)-t^*\hat\theta\big]\Big)\big(1+O(n^{-1})\big),
\quad \psi'(t^*)=\hat\theta.
\label{eq:saddlepoint}
\end{equation}
Since $\psi(t^*)-t^*\hat\theta=-I(\hat\theta)$, the exponent is $-nI(\hat\theta)$, bridging large deviation exponents and the polynomial prefactors entering the moderate deviation regime.
\end{remark}

\section{Risk-optimal threshold}\label{sec:main-result}

\subsection{Heuristic risk balancing}\label{sec:risk-min}

This section presents an intuitive, heuristic derivation of the moderate deviation threshold. The rigorous result with detailed asymptotics is given in Theorem~\ref{thm:main}.

The integrated Bayes risk under $0$--$1$ loss decomposes as
\begin{equation}
R_n(c)=\pi_0\,\Pbb_{\theta_0}\!\bigl(|t_n|>c\bigr)
  +\pi_a\int_{\theta\neq\theta_0}\Pbb_\theta\!\bigl(|t_n|\le c\bigr)\,\pi(\theta)\,\diff\theta.
\label{eq:risk-explicit}
\end{equation}
The type~I term is evaluated by the moderate deviation expansion (Lemma~\ref{lem:uniform-md}):
\[
\pi_0\,\Pbb_{\theta_0}(|t_n|>c)\sim\frac{\pi_0\sqrt{2}}{\sqrt{\pi}\,c}\,e^{-c^2/2},
\qquad c\to\infty.
\]
The type~II term is evaluated by localization of the prior near $\theta_0$. At the alternative $\theta$ with $|\theta-\theta_0|\le\sigma c/\sqrt{n}$, the posterior concentrates on $\bar{x}$ and $\Pbb_\theta(|t_n|\le c)\approx 1$ by concentration. Thus
\begin{multline*}
\pi_a\int\Pbb_\theta(|t_n|\le c)\,\pi(\theta)\,\diff\theta
\approx\pi_a\int_{|\theta-\theta_0|\le\sigma c/\sqrt{n}}\pi(\theta)\,d\theta\\
\approx 2\pi_a c_\pi\cdot\frac{\sigma c}{\sqrt{n}}\,(1+o(1)),
\qquad c\to\infty,\ n\to\infty.
\end{multline*}
Setting $\partial R_n/\partial c=0$ to balance type~I and type~II contributions:
\[
\frac{\partial}{\partial c}\left[\pi_0\sqrt{\tfrac{2}{\pi}}\,\frac{e^{-c^2/2}}{c}(1+o(1))\right]
=2\pi_a c_\pi\sigma\,n^{-1/2}(1+o(1)).
\]
At leading exponential order, the type~I rate is governed by $e^{-c^2/2}$ (the polynomial prefactors in $c$ do not affect the exponential balance). Matching exponential orders:
\[
e^{-c^2/2}\asymp n^{-1/2}\qquad\Longleftrightarrow\qquad c^2\asymp\log n.
\]
This heuristic argument identifies the correct scaling but does not determine the threshold constant or remainder; we call the resulting boundary ``risk-optimal'' only after the rigorous derivation in Theorem~\ref{thm:main}, which obtains $c^2=\log n+O(1)$ with explicit constants via Laplace expansion of the Bayes factor.

\subsection{Main theorem}\label{sec:theorem}

\begin{theorem}[Risk-optimal threshold]\label{thm:main}
Consider testing $H_0:\theta=0$ versus $H_a:\theta\neq 0$ with $\bar{X}_n\sim\mathcal{N}(\theta,\sigma^2/n)$ and $\sigma$ known. Let the prior under $H_a$ satisfy Assumption~\ref{ass:prior} with $\pi(0)=c_\pi>0$. Under $0$--$1$ loss with prior weights $\pi_0,\pi_a>0$, the Bayes factor in favour of $H_0$ satisfies
\begin{equation}
\begin{split}
BF_{01}(\bar{x})&=\frac{\sqrt{n}}{\sigma\sqrt{2\pi}\,c_\pi}\,
e^{-t^2/2}\\
&\quad\times\bigl(1+O\bigl(\!\sqrt{(\log n)/n}\bigr)\bigr),
\end{split}
\label{eq:bf-expansion}
\end{equation}
where $t=\sqrt{n}\,\bar{x}/\sigma$,
uniformly for~$\bar{x}$ on the moderate deviation scale.
Setting $BF_{01}=\pi_a/\pi_0$ and solving yields
\begin{equation}
\begin{split}
t_{\mathrm{crit}}^2 &= \log n + \log(c_\pi^{-2}) - \log(2\pi\sigma^2)\\
&\quad + 2\log(\pi_0/\pi_a) + O\bigl(\!\sqrt{(\log n)/n}\bigr).
\end{split}
\label{eq:threshold}
\end{equation}
\end{theorem}

Three structural features of \eqref{eq:threshold} merit emphasis. The leading behaviour $t_{\mathrm{crit}}^2\sim\log n$ is universal: it reflects the balance between the $\sqrt{n}$ prefactor and the Gaussian tail $\exp(-t^2/2)$ in \eqref{eq:bf-expansion}, and depends on neither the prior family nor the parametric model. The additive constant, by contrast, is prior-sensitive, determined by the local density $c_\pi=\pi(0)$; a larger~$\pi(0)$ shifts the indifference boundary by~$\log(c_\pi^{-2})$. Finally, the remainder is
$O\!\bigl(\sqrt{(\log n)/n}\bigr)$
under the stated regularity, so the displayed constants are identified up to a vanishing deterministic correction on the moderate deviation band.

\begin{proof}
\textit{Step~1: Exact marginal representation.}
Under $H_a$, the marginal likelihood is
\[
m_a(\bar{x})=\int f(\bar{x}\mid\theta)\,\pi(\theta)\,d\theta
=\frac{\sqrt{n}}{\sigma\sqrt{2\pi}}\int
\exp\!\left(-\frac{n(\bar{x}-\theta)^2}{2\sigma^2}\right)\pi(\theta)\,d\theta.
\]
Substituting $u=\sqrt{n}(\theta-\bar{x})/\sigma$:
\[
m_a(\bar{x})=\frac{1}{\sqrt{2\pi}}\int_{\R} e^{-u^2/2}\,
\pi\!\left(\bar{x}+\frac{\sigma u}{\sqrt{n}}\right)du.
\]

\textit{Step~2: Good-set / bad-set decomposition.}
We localize the integral by splitting into a central region where a uniform Taylor expansion is valid, and a tail region controlled by Gaussian decay. Let $M_n\to\infty$ be any sequence satisfying $M_n/\sqrt{n}\to 0$, and define
\[
G_n=\{|u|\le M_n\},\qquad B_n=\{|u|>M_n\}.
\]
The specific rate at which $M_n$ diverges is not important; any sequence satisfying these conditions suffices. For instance, one may take $M_n=n^{\alpha}$ for any $0<\alpha<1/2$ (e.g., $\alpha=1/8$).

Write $m_a(\bar{x})=I_G(\bar{x})+I_B(\bar{x})$, where
\begin{align*}
I_G(\bar{x})&=\frac{1}{\sqrt{2\pi}}\int_{G_n} e^{-u^2/2}\,\pi\!\left(\bar{x}+\frac{\sigma u}{\sqrt{n}}\right)du,\\
I_B(\bar{x})&=\frac{1}{\sqrt{2\pi}}\int_{B_n} e^{-u^2/2}\,\pi\!\left(\bar{x}+\frac{\sigma u}{\sqrt{n}}\right)du.
\end{align*}

\textit{Step~3: Uniform Taylor expansion on the good set.}
Since $M_n/\sqrt{n}\to 0$, the argument $\bar{x}+\sigma u/\sqrt{n}$ remains in a shrinking neighbourhood of zero uniformly over $|u|\le M_n$. Indeed, on the moderate deviation band $|\bar{x}|=O(\sqrt{(\log n)/n})$,
\[
\left|\bar{x}+\frac{\sigma u}{\sqrt{n}}\right|
\le |\bar{x}|+\frac{\sigma M_n}{\sqrt{n}}
= O\!\left(\sqrt{\frac{\log n}{n}}\right)+O\!\left(\frac{M_n}{\sqrt{n}}\right)\to 0.
\]
Hence, for all sufficiently large $n$, the argument of $\pi$ remains in a fixed compact neighbourhood of $0$ on which $\pi'''$ is bounded (since $\pi\in C^3(\R)$ by Assumption~\ref{ass:prior}). The second-order Taylor expansion is therefore valid uniformly on $G_n$:
\[
\pi\!\left(\bar{x}+\frac{\sigma u}{\sqrt{n}}\right)
=\pi(\bar{x})+\pi'(\bar{x})\frac{\sigma u}{\sqrt{n}}
+\frac{1}{2}\pi''(\bar{x})\frac{\sigma^2 u^2}{n}+R_n(u,\bar{x}),
\]
with $|R_n(u,\bar{x})|\le C|u|^3/n^{3/2}$ uniformly for $|u|\le M_n$ and $\bar{x}$ on the moderate deviation band.

Integrating term by term against $e^{-u^2/2}/\sqrt{2\pi}$ over $G_n$: the first-order term vanishes by symmetry ($\int_{G_n} u\,e^{-u^2/2}\,du=0$). For the zeroth- and second-order terms,
\begin{align*}
\frac{1}{\sqrt{2\pi}}\int_{G_n}e^{-u^2/2}\,du&=1+O(e^{-M_n^2/2}),\\
\frac{1}{\sqrt{2\pi}}\int_{G_n}u^2\,e^{-u^2/2}\,du&=1+O(M_n\,e^{-M_n^2/2}).
\end{align*}
The remainder satisfies
\[
\left|\frac{1}{\sqrt{2\pi}}\int_{G_n}R_n(u,\bar{x})\,e^{-u^2/2}\,du\right|
\le \frac{C}{n^{3/2}}\int_{\R}|u|^3\,\frac{e^{-u^2/2}}{\sqrt{2\pi}}\,du=O(n^{-3/2}).
\]
Therefore
\begin{equation}
I_G(\bar{x})=\pi(\bar{x})+\frac{\sigma^2}{2n}\pi''(\bar{x})+O(n^{-3/2})+O(e^{-M_n^2/2}),
\label{eq:good-set}
\end{equation}
uniformly on the moderate deviation band.

\textit{Step~4: Bounding the bad set.}
By the polynomial-growth condition in Assumption~\ref{ass:prior}, $|\pi(\theta)|\le C(1+|\theta|^m)$, so
\[
\left|\pi\!\left(\bar{x}+\frac{\sigma u}{\sqrt{n}}\right)\right|
\le C\bigl(1+|\bar{x}|^m+n^{-m/2}|u|^m\bigr).
\]
Therefore
\[
|I_B(\bar{x})|\le \frac{C}{\sqrt{2\pi}}\int_{|u|>M_n}(1+|u|^m)\,e^{-u^2/2}\,du
=o(n^{-K})
\]
for every fixed $K>0$, since $M_n\to\infty$ and Gaussian tails vanish faster than any polynomial rate. In particular, $I_B(\bar{x})=o(n^{-1})$ uniformly on the moderate deviation band.

\textit{Step~5: Recombination and prior localization.}
Combining Steps~3 and~4:
\[
m_a(\bar{x})=\pi(\bar{x})+\frac{\sigma^2}{2n}\pi''(\bar{x})+o(n^{-1}),
\]
uniformly for $|\bar{x}|=O(\sqrt{(\log n)/n})$.
Since $\pi(0)=c_\pi>0$ and $\pi$ is continuous,
$\pi(\bar{x})=c_\pi+O(|\bar{x}|)=c_\pi\bigl(1+O(\sqrt{(\log n)/n})\bigr)$.
Because $n^{-1}=o(\sqrt{(\log n)/n})$, the dominant error is the prior-localization term, so
\[
m_a(\bar{x})=c_\pi\bigl(1+O\!\bigl(\sqrt{(\log n)/n}\bigr)\bigr),
\]
uniformly for $|\bar{x}|=O(\sqrt{(\log n)/n})$.

\textit{Step~6: Bayes factor and threshold.}
The null likelihood is $f(\bar{x}\mid 0)$
${}=(\sqrt{n}/\sigma\sqrt{2\pi})\,e^{-t^2/2}$
with~$t=\sqrt{n}\,\bar{x}/\sigma$. Dividing:
\begin{align*}
BF_{01}(\bar{x})&=\frac{f(\bar{x}\mid 0)}{m_a(\bar{x})}
=\frac{\sqrt{n}}{\sigma\sqrt{2\pi}\,c_\pi}\,e^{-t^2/2}\\
&\quad\times\bigl(1+O\bigl(\!\sqrt{(\log n)/n}\bigr)\bigr),
\end{align*}
which is equation \eqref{eq:bf-expansion}. The expansion is deterministic and uniform in $\bar{x}$ on the moderate deviation band.

Setting $BF_{01}=\pi_a/\pi_0$ and taking logarithms:
\[
\frac{t^2}{2}=\frac{1}{2}\log n-\log(\sigma c_\pi\sqrt{2\pi})
+\log(\pi_0/\pi_a)+O\bigl(\!\sqrt{(\log n)/n}\bigr).
\]
Multiplying both sides by $2$ and expanding $-2\log(\sigma c_\pi\sqrt{2\pi})=-2\log c_\pi-\log\sigma^2-\log(2\pi)=\log(c_\pi^{-2})-\log(2\pi\sigma^2)$ gives
\[
t^2=\log n+\log(c_\pi^{-2})-\log(2\pi\sigma^2)+2\log(\pi_0/\pi_a)
+O\!\bigl(\sqrt{(\log n)/n}\bigr),
\]
which is \eqref{eq:threshold}. The threshold is Bayes-risk-optimal because it is the Bayes decision boundary: under $0$--$1$ loss, the Bayes rule rejects $H_0$ when $BF_{01}<\pi_a/\pi_0$, and this rule minimizes integrated Bayes risk among all tests.
\end{proof}

\medskip
\noindent\textit{Relation to stochastic marginal-likelihood asymptotics.}
Theorem~\ref{thm:main} is a local analytical statement: it provides a uniform expansion of the Bayes factor as a function of the observed statistic $\bar{x}$ over the moderate deviation band $|\bar{x}|=O(\sqrt{(\log n)/n})$, and this is what yields the explicit threshold in \eqref{eq:threshold}. By contrast, Appendix~\ref{app:dawid} records a Dawid-style stochastic asymptotic for log marginal-likelihood ratios under repeated sampling. The two results are consistent but serve different purposes: the appendix theorem describes typical sample-path behaviour, whereas Theorem~\ref{thm:main} gives the sharper deterministic control needed for the threshold calculation.

\subsection{Interpretation}\label{sec:interpretation}

\noindent\textit{Why the logarithmic boundary arises.}
The mechanism behind the $\sqrt{\log n}$ threshold can be understood through three complementary lenses, all leading to the same asymptotic balance.

The first is an evidence-versus-complexity argument. From \eqref{eq:bf-expansion},
$BF_{10}\propto c_\pi\exp(t^2/2)/\sqrt{n}$:
the exponential captures the likelihood gain while $1/\sqrt{n}$ is the complexity penalty from integrating over the alternative. The critical threshold is where the two balance,
$\exp(t^2/2)\sim\sqrt{n}$,
giving $t^2\sim\log n$.

The second is a risk-balancing argument. Type~I risk decays as $\exp(-c^2/2)$; type~II risk decays as~$n^{-1/2}$ from posterior concentration. Equating the two rates gives
$c^2=\log n+O(1)$,
as in \S\ref{sec:risk-min}. A constant threshold leaves type~I error undecayed; a threshold growing faster than~$\sqrt{\log n}$ lets type~II error dominate. The logarithmic rate is the unique balancing point.

The third is a BIC connection. The Laplace approximation introduces a penalty of $(d/2)\log n$ relative to the maximized likelihood; the testing threshold introduces $t^2_{\rm crit}\sim\log n$ relative to the standardized statistic. Both are manifestations of the same polynomial prefactor in the Laplace integral: Gaussian integrals evaluated at the moderate deviation scale produce polynomial rather than exponential corrections. The BIC penalty governs model selection; the $\sqrt{\log n}$ threshold governs evidence accumulation.

\medskip
\noindent\textit{Universality and prior sensitivity.}
A key feature of \eqref{eq:threshold} is the separation between universal and prior-sensitive terms. The leading term $\log n$ depends on neither the prior family nor the parametric model; it is a consequence of the $\sqrt{n}$ scaling of Fisher information in any regular model. The additive constants, by contrast, depend on the local prior density $c_\pi=\pi(\theta_0)$, the Fisher information $\sigma^{-2}$, and the prior model odds $\pi_0/\pi_a$. These constants shift the threshold but do not affect its growth rate. This universality is what allows the same asymptotic framework to encompass Jeffreys' threshold, BIC, and the Chernoff--Stein exponents simultaneously.

\medskip
\noindent\textit{Lindley's paradox.}
The divergence between the risk-optimal threshold $\sqrt{\log n}$ and the classical fixed-$\alpha$ critical value $z_{1-\alpha/2}$ grows without bound. At $n=1{,}000$ with $t=1.96$, the Bayes factor is
\[
BF_{01}=\sqrt{\tfrac{\pi\cdot 1000}{2}}\cdot e^{-1.96^2/2}\approx 39.6\times 0.147\approx 5.8,
\]
which constitutes evidence \emph{for} $H_0$. The risk-optimal threshold at $n=1{,}000$ is $t_{\mathrm{crit}}=\sqrt{\log(\pi\cdot 1000/2)}\approx 2.71$; the observed $t=1.96$ falls below this boundary. The paradox is not a pathology but a direct consequence of the logarithmic correction absent from classical calibration \citep{bartlett1957comment,lindley1957,jeffreys1961}.

\medskip
\noindent\textit{Sample-size adaptivity.}
Table~\ref{tbl:jeffreys-cutoffs} displays the risk-optimal critical values for the Cauchy prior $\theta\sim C(0,\sigma)$, giving $c_\pi=1/(\pi\sigma)$ and $t_{\mathrm{crit}}^2=\log(\pi n/2)+o(1)$. For small $n$ ($n\lesssim 50$), the Bayesian threshold falls below $1.96$, permitting rejection on weaker evidence; for large $n$, it exceeds $1.96$, requiring stronger evidence.

\begin{table}[ht]
\centering
\begin{tabular}{rcc}
\toprule
$n$ & Bayes $|t|$-cutoff $\sqrt{\log(\pi n/2)}$ & Classical $p$-value at this $|t|$ \\
\midrule
5 & 1.43 & 0.152 \\
10 & 1.66 & 0.097 \\
100 & 2.25 & 0.024 \\
1{,}000 & 2.71 & 0.007 \\
100{,}000 & 3.46 & 0.0005 \\
\bottomrule
\end{tabular}
\caption{Risk-optimal critical values from Theorem~\ref{thm:main} under a Cauchy prior with equal prior odds ($\pi_0=\pi_a=1/2$) and $\sigma=1$. The formula
$t_{\mathrm{crit}}=\sqrt{\log(\pi n/2)}$ includes the prior normalization constant $c_\pi=1/(\pi\sigma)$.}
\label{tbl:jeffreys-cutoffs}
\end{table}

\section{Extensions}\label{sec:extensions}

\subsection{Separated alternatives and error-exponent asymptotics}\label{sec:chernoff}

For simple-versus-simple testing $H_0:\theta=\theta_0$ vs.\ $H_1:\theta=\theta_1$,
\citet{Chernoff1952_MeasureAsymptoticEfficiency} established that Bayes risk converges to zero at an exponential rate independent of the prior at first order. \citet{EfronTruax1968} refined this to
\begin{equation}
P_n(\text{error})=\frac{K}{\sqrt{n}}\exp(-nD_C)\,(1+O(n^{-1})),
\label{eq:efron-truax}
\end{equation}
where $D_C=-\min_{0\le s\le 1}\log\int f_{\theta_0}(x)^{1-s}f_{\theta_1}(x)^s\,dx$ is the Chernoff information and $K>0$ depends on the prior and geometry at the Chernoff point.

\begin{remark}[Gaussian case]
For $H_0:\mu=0$ versus $H_1:\mu=\delta$ with $\sigma=1$, the Chernoff point is $\mu^*=\delta/2$, the Chernoff information is $D_C=\delta^2/8$, and
\[
P_n(\text{error})=\frac{2}{\delta\sqrt{2\pi n}}\,e^{-n\delta^2/8}.
\]
\end{remark}

More generally, \citet{EfronTruax1968} show that
\begin{equation}
P_n(\text{error})=\frac{2}{\sqrt{2\pi n}}\sqrt{\frac{V(\mu^*)}{(\mu_1-\mu_0)^2}}\,
e^{-nD_C}(1+O(n^{-1})),
\label{eq:efron-truax-general}
\end{equation}
where $V(\mu^*)$ is the variance at the Chernoff point. The Chernoff--Stein lemma gives the rate:
\[
\lim_{n\to\infty}\frac{1}{n}\log\bigl(p\,\alpha_n+(1-p)\,\beta_n\bigr)=-D_C.
\]
In composite testing, additional logarithmic terms appear; the BIC penalty $\frac{d}{2}\log n$ is a manifestation of this correction.

Classical Chernoff--Stein theory therefore captures the fixed-separated regime, where the hypotheses remain a constant Kullback--Leibler distance apart and the dominant asymptotics are exponential in $n$. The present paper studies a different regime: the alternative approaches the null with $n$, so the Bayes factor enters a moderate deviation zone in which $t_n\asymp\sqrt{\log n}$ and the resulting tail probabilities decay polynomially rather than at a fixed exponential rate. From this viewpoint, the moderate deviation boundary does not replace the error-exponent picture; it refines it in the local composite-testing regime where the mixture structure under $H_a$ becomes first-order relevant. This is also the natural point of contact with Hoeffding's large-deviation analysis of testing tradeoffs, which belongs to the same fixed-alternative lineage but leaves room for the logarithmic correction derived here once the alternative is integrated and allowed to shrink toward the null.

\medskip
\noindent\textit{One-sided Gaussian testing.}
For $H_0:\theta=\theta_0$ versus $H_1:\theta=\theta_1$ with $\theta_1>\theta_0$ and $X_i\sim N(\theta,\sigma^2)$, the Bayes decision rule with equal priors rejects $H_0$ when $\bar{X}_n>(\theta_0+\theta_1)/2$. The type~I error satisfies
\[
\alpha_n\sim\frac{2\sigma}{\Delta\sqrt{2\pi n}}
\exp\!\left(-\frac{n\Delta^2}{8\sigma^2}\right),\qquad\Delta=\theta_1-\theta_0,
\]
and by symmetry $\beta_n\sim\alpha_n$, confirming $D_C=\Delta^2/(8\sigma^2)$ with $s^*=1/2$.

\subsection{Composite alternatives and mixture interpretation}\label{sec:mixture}

Under Bayesian testing, the alternative hypothesis is represented by the mixture model
\[
m_1(\bm{x})=\int f(\bm{x}\mid\theta)\,\pi(\theta)\,\diff\theta.
\]
This makes the problem intrinsically composite even when the null is a single point. The mixture representation is not merely a formal rewriting: it is the source of the threshold inflation. Relative to the point null, the alternative spreads prior mass over a local $n^{-1/2}$ neighbourhood of $\theta_0$, and the Laplace approximation converts this spread into the $n^{-1/2}$ normalization term appearing in Theorem~\ref{thm:main}. The risk-optimal threshold marks the point at which the likelihood gain
$\exp(t^2/2)$ overcomes the complexity of integrating over the local alternative family.

This provides a composite-testing interpretation of the moderate deviation boundary. In simple-versus-simple Chernoff--Stein theory, first-order behaviour is governed by the error exponent alone. Under a mixture alternative, the polynomial prefactor induced by local integration is no longer negligible, and it is precisely this prefactor that produces the logarithmic correction
$t_{\mathrm{crit}}^2=\log n+O(1)$.
The resulting threshold may thus be read as a local, composite-alternative refinement of the classical error-exponent picture.

\subsection{The Rubin--Sethuraman risk calculus}\label{sec:rs}

Following \citet{rubin1965probabilities,RubinSethuraman1965_BRE}, suppose $\{(Y_i,Z_i)\}_{i=1}^n$ are i.i.d.\ with
\[
\begin{pmatrix}Y_i\\[2pt]Z_i\end{pmatrix}\sim\mathcal{N}\!\left(\begin{pmatrix}\theta\\[2pt]\psi\end{pmatrix},\,\sigma^2 I_{k+m}\right),
\]
where $\theta\in\mathbb{R}^k$ is the parameter of interest and $\psi\in\mathbb{R}^m$ the nuisance. With local prior$\times$loss weight
$g(\theta)=\gamma(\theta)\,\vnorm{\theta}^\lambda\,h(\theta/\vnorm{\theta})$ ($\lambda>-k$),
the RS threshold is
\begin{equation}
\|\hat\theta\|=\sqrt{\frac{\log n}{n}}\,(\lambda+k+c_n)^{1/2},\qquad c_n=o(1).
\label{eq:rs-threshold}
\end{equation}

The Bayes risk splits as $R_n=R_1+R_2$ with
\[
R_2\sim C\cdot(\log n)^{(k+\lambda)/2}\cdot n^{-(\lambda+k)/2}\cdot\int_{S^{k-1}}h(\omega)\,d\omega,
\]
while $R_1\approx C_1 n^{-a^2/(2\sigma^2)}/\sqrt{\log n}$ by Lemma~\ref{lem:uniform-md}.

\medskip
\noindent\textit{Sparse Gaussian setting.}
The $\sqrt{\log n}$ scaling has a counterpart in high-dimensional sparse inference. Under the two-groups model with sparsity $p\to 0$, \citet{datta2013asymptotic} show that the horseshoe decision rule achieves the Bayes oracle risk up to $O(1)$ when the effective threshold grows as $\sqrt{\log(1/p)}$, the sparse analogue of the RS boundary, with sparsity $1/p$ replacing sample size $n$.

\medskip
\noindent\textit{Prior tail behaviour.}
If $\pi(0)\in(0,\infty)$, then $BF_{01}\propto\sqrt{n}\,e^{-t^2/2}$ as in Theorem~\ref{thm:main}. The horseshoe prior \citep{carvalho2010horseshoe}, satisfying $\pi(\theta)\asymp\log(1+1/\theta^2)$ as $|\theta|\to 0$, modifies the asymptotics:
$\log BF_{01}\sim\frac{1}{2}\log n-\log(\log n)+C-n\bar{x}^2/(2\sigma^2)$,
giving rejection threshold $n\bar{x}^2/\sigma^2\gtrsim\log n-2\log\log n+O(1)$. This slower growth reflects the prior's heavier tails near zero, which allocate more mass to small alternatives and thereby dilute evidence more aggressively than a flat prior.

\subsection{One-parameter exponential families}\label{sec:exponential}

The logarithmic threshold scaling extends beyond Gaussian models via local asymptotic normality. Consider a one-parameter exponential family with densities
$f(x\mid\theta)=h(x)\exp\{\theta T(x)-\psi(\theta)\}$
and Fisher information $I(\theta)=\psi''(\theta)$.

\begin{assumption}[Exponential family regularity]\label{ass:expfam}
(i)~$\psi(\theta)$ is thrice differentiable with $I(\theta_0)=\psi''(\theta_0)>0$.
(ii)~$D(P_{\theta_0+h}\|P_{\theta_0})=I(\theta_0)h^2/2+O(h^3)$.
(iii)~Local asymptotic normality holds:
$\log\frac{f(X^n|\theta_0+h/\sqrt{n})}{f(X^n|\theta_0)}
=h\Delta_n-\frac{h^2 I(\theta_0)}{2}+o_\Pbb(1)$,
where $\Delta_n=\sqrt{n}(\bar{T}_n-\psi'(\theta_0))$ is the normalized score.
\end{assumption}

\begin{theorem}[Moderate deviation threshold in exponential families]\label{thm:expfam}
Under Assumptions~\ref{ass:prior} and~\ref{ass:expfam}, the Bayes factor for testing $H_0:\theta=\theta_0$ versus $H_a:\theta\neq\theta_0$ satisfies
\[
BF_{01}(\bar{T}_n)=\frac{\sqrt{nI(\theta_0)}}{\sqrt{2\pi}\,c_\pi}\,
\exp\!\left(-\frac{S_n^2}{2}\right)\!\left(1+O\!\left(\sqrt{\tfrac{\log n}{n}}\right)\right),
\]
where $S_n=\sqrt{n}\,(\bar{T}_n-\psi'(\theta_0))/\sqrt{I(\theta_0)}$ is the standardized score statistic and $c_\pi=\pi(\theta_0)>0$. The risk-optimal rejection threshold satisfies
\[
S_n^2=\log n+\log\!\bigl(c_\pi^{-2}\bigr)+\log I(\theta_0)-\log(2\pi)+2\log\frac{\pi_0}{\pi_a}+O\!\left(\sqrt{\tfrac{\log n}{n}}\right).
\]
\end{theorem}

\begin{proof}
\textit{Step~1: Local asymptotic normality.}
By Assumption~\ref{ass:expfam}(iii), the log-likelihood ratio under $H_a$ is
\[
\log\frac{f(X^n|\theta_0+h/\sqrt{n})}{f(X^n|\theta_0)}=h\Delta_n-\frac{h^2 I(\theta_0)}{2}+o_\Pbb(1),
\]
where $\Delta_n=\sqrt{n}(\bar{T}_n-\psi'(\theta_0))$ is the normalized score and $I(\theta_0)=\psi''(\theta_0)$ is the Fisher information.

\textit{Step~2: Marginal likelihood integral under LAN.}
With substitution $\theta=\theta_0+u/\sqrt{n}$, the marginal likelihood under $H_a$ is
\[
m_a(X^n)=f(X^n|\theta_0)\int\exp\left(u\Delta_n-\frac{u^2 I(\theta_0)}{2}\right)\pi(\theta_0+u/\sqrt{n})\,du.
\]

\textit{Step~3: Localization and completion.}
By Assumption~\ref{ass:prior}, $\pi(\theta_0+u/\sqrt{n})=c_\pi(1+o(1))$ in the $u$-region where the exponential and prior are nonnegligible. Completing the square in the exponent:
\[
u\Delta_n-\frac{I(\theta_0)\,u^2}{2}
=-\frac{I(\theta_0)}{2}\!\left(u-\frac{\Delta_n}{I(\theta_0)}\right)^{\!2}
+\frac{\Delta_n^2}{2I(\theta_0)}.
\]
Extracting the constant $\exp(\Delta_n^2/(2I(\theta_0)))$ and evaluating the remaining Gaussian integral $\int\exp(-I(\theta_0)(u-\Delta_n/I(\theta_0))^2/2)\,du=\sqrt{2\pi/I(\theta_0)}$, while accounting for the substitution $d\theta=n^{-1/2}\,du$, gives
\[
m_a(X^n)=f(X^n|\theta_0)\cdot\frac{c_\pi\sqrt{2\pi}}{\sqrt{nI(\theta_0)}}\cdot
\exp\!\left(\frac{\Delta_n^2}{2I(\theta_0)}\right)\!\left(1+O\!\left(\sqrt{\tfrac{\log n}{n}}\right)\right).
\]
Since $S_n=\Delta_n/\sqrt{I(\theta_0)}$, we have $\Delta_n^2/(2I(\theta_0))=S_n^2/2$.

\textit{Step~4: Bayes factor and threshold.}
Therefore
\[
BF_{01}=\frac{\sqrt{nI(\theta_0)}}{\sqrt{2\pi}\,c_\pi}\,
\exp\!\left(-\frac{S_n^2}{2}\right)\!\left(1+O\!\left(\sqrt{\tfrac{\log n}{n}}\right)\right).
\]
Setting $BF_{01}=\pi_a/\pi_0$ and taking logarithms:
\[
\frac{S_n^2}{2}=\frac{1}{2}\log n-\log(c_\pi\sqrt{2\pi/I(\theta_0)})+\log\frac{\pi_0}{\pi_a}+O\!\left(\sqrt{\tfrac{\log n}{n}}\right),
\]
yielding the explicit threshold
\[
S_n^2=\log n+\log\!\bigl(c_\pi^{-2}\bigr)+\log I(\theta_0)-\log(2\pi)+2\log\frac{\pi_0}{\pi_a}+O\!\left(\sqrt{\tfrac{\log n}{n}}\right),
\]
which mirrors \eqref{eq:threshold} with $\sigma^{-2}$ replaced by $I(\theta_0)$.
\end{proof}

\section{Comparison with alternative calibrations}\label{sec:comparison}

We compare three calibration methods algebraically.

\subsection{Fixed-$\alpha$ testing}

The Neyman--Pearson approach fixes $\alpha=0.05$ regardless of $n$: the critical value $z_{1-\alpha/2}=1.96$ is independent of sample size. Under integrated Bayes risk with fixed prior weights, the fixed-$\alpha$ test is asymptotically suboptimal. The risk-optimal procedure achieves $B_n^\star=\pi_0\alpha_n^\star+\pi_a\beta_n^\star$ with $\alpha_n^\star\to 0$ as $n^{-a^2/(2\sigma^2)}$ (Lemma~\ref{lem:uniform-md}), yielding strictly lower integrated risk for all sufficiently large $n$.

\subsection{E-value thresholds}

E-values \citep{shafer2021testing,grunwald2020safe} are nonneg\-ative random variables $E_n$ satisfying $\Ebb_{\theta_0}[E_n]\le 1$. Safe testing rejects when $E_n>1/\alpha$, giving
\[
c_{\mathrm{EV}}=\sqrt{2\log(1/\alpha)},\qquad c_{\mathrm{RS}}=\sqrt{\log n+O(1)}.
\]
The ratio $c_{\mathrm{EV}}/c_{\mathrm{RS}}=\sqrt{2\log(1/\alpha)/(\log n+O(1))}\to 0$ as $n\to\infty$.
The e-value threshold is $O(1)$ in $n$; it does not grow with sample size. Safe testing controls type~I error uniformly over stopping times via Ville's inequality but does not minimize integrated Bayes risk. \citet{PolsonSokolovZantedeschi2026_Evalues} develop a typed framework separating sequential evidence into representation, validity, and decision layers, clarifying when the e-value and Bayes risk calibrations coincide and when they diverge.

\medskip
\noindent\textit{Summary.}
\begin{align*}
\text{Risk-optimal (RS):} &\quad t_{\mathrm{crit}}=\sqrt{\log n+O(1)}\quad\text{(grows with }n\text{)},\\
\text{Fixed-}\alpha\text{ (NP):} &\quad t_{\mathrm{crit}}=z_{1-\alpha/2}=O(1)\quad\text{(constant in }n\text{)},\\
\text{E-value (safe):} &\quad t_{\mathrm{crit}}=\sqrt{2\log(1/\alpha)}=O(1)\quad\text{(constant in }n\text{)}.
\end{align*}
Under integrated Bayes risk with fixed prior weights and fixed non-adaptive thresholds, only the risk-optimal calibration operates on the moderate deviation scale. The divergence between this threshold and the fixed-$\alpha$ critical value as $n\to\infty$ is the mechanism underlying the Lindley paradox.

\medskip
\noindent\textit{Scope of the comparison.}
The comparison above assumes fixed (non-adaptive) rejection thresholds and evaluates asymptotic scaling in $n$. Alternative calibration strategies that adapt thresholds to the data, such as empirical Bayes or posterior predictive methods, could in principle produce different scaling behaviour, though the integrated risk criterion would still select the moderate deviation boundary under the stated regularity conditions. The comparison is algebraic, not normative: each calibration framework serves a different inferential goal (risk minimization, type~I control, anytime validity), and the scaling differences reflect these distinct objectives.

\section{Connections to the literature}\label{sec:connections}

The moderate deviation boundary \eqref{eq:scales} appears, in various guises, across several lines of work. We organize the connections into four strands, each arriving at the same asymptotic regime from a different starting point.

\medskip
\noindent\textit{Jeffreys and the Lindley paradox.}
\citet{Jeffreys1935} first observed that the Bayes factor threshold for point-null testing grows with sample size, noting that evidence thresholds expressed in terms of the standardized statistic should increase as data accumulate. \citet{lindley1957} sharpened this into a paradox: for any fixed significance level $\alpha$, there exists a sample size $n^*$ beyond which a $p$-value of $\alpha$ corresponds to a Bayes factor favouring $H_0$. The resolution, as formalized in Theorem~\ref{thm:main}, is that the risk-optimal threshold grows as $\sqrt{\log n}$, so the classical fixed-$\alpha$ critical value eventually falls below the Bayesian decision boundary. The Lindley paradox is thus a direct consequence of the logarithmic correction, not a deficiency of either framework.

\medskip
\noindent\textit{Rubin--Sethuraman moderate deviations.}
\citet{rubin1965probabilities,RubinSethuraman1965_BRE} developed the systematic moderate deviation analysis of Bayes risk, establishing the $\sqrt{(\log n)/n}$ scale for rejection boundaries under integrated risk with general prior structures. Their analysis identified the polynomial decay of tail probabilities at this scale and showed that the Bayes risk is minimized at the moderate deviation boundary. The present paper states their programme under explicit regularity conditions and derives the sharp threshold constant.

\medskip
\noindent\textit{Laplace approximation and BIC asymptotics.}
The BIC penalty $(d/2)\log n$ arises from the Laplace approximation to the marginal likelihood, in which the Gaussian integral over the posterior produces a prefactor of order $n^{-d/2}$. This is the same polynomial correction that appears in the testing threshold: the BIC penalty governs the additive log-marginal cost of model complexity, while the $\sqrt{\log n}$ threshold governs the critical statistic for evidence accumulation. \citet{dawid1984statistical} established prequential asymptotics for marginal-likelihood ratios under repeated sampling, while \citet{Clarke1990} refined Laplace expansions to identify the prior-density and Fisher-information constants entering the BIC correction. Theorem~\ref{thm:main} may be read as a local threshold-level counterpart of these results: instead of a stochastic expansion for the log evidence along random sample paths, it provides a deterministic uniform expansion of the Bayes factor over the moderate deviation region where the decision boundary is formed.

\medskip
\noindent\textit{Information-theoretic error exponents.}
The Chernoff--Stein theory \citep{Chernoff1952_MeasureAsymptoticEfficiency} governs the exponential rate $\exp(-nD_C)$ in simple-versus-simple testing, where the threshold is prior-insensitive at first order. The Efron--Truax refinement \citep{EfronTruax1968} supplies the polynomial prefactor $K/\sqrt{n}$ that becomes relevant at the moderate deviation scale. The Bahadur--Rao theorem \citep{bahadur1960deviations} provides the analogous refinement for tail probabilities of sample means. Hoeffding's large-deviation analysis of hypothesis testing belongs to the same lineage, emphasizing the full error-exponent tradeoff under separated alternatives. In composite testing, however, the integration over the alternative introduces the additional logarithmic correction, connecting the Chernoff--Stein exponent and Hoeffding-type error exponents to the BIC penalty and the $\sqrt{\log n}$ threshold.

\medskip
\noindent\textit{Calibration and reconciliation.}
\citet{sellke2001calibration} showed that a $p$-value of $0.05$ corresponds to posterior probabilities of $H_0$ no smaller than $0.22$--$0.41$, depending on the prior class. \citet{LopesPolson2019_Redux} provide modern calibrations linking $p$-values, Bayes factors, and posterior odds. \citet{datta2013asymptotic} extended the RS programme to global--local priors in high-dimensional settings.

Collectively, these strands converge on the moderate deviation scale as the unifying asymptotic lens. The present paper provides a common asymptotic coordinate system, namely the three equivalent scales \eqref{eq:scales}, on which these classical results can be simultaneously situated.

\section{Discussion}\label{sec:disc}

\noindent\textit{Scope and limitations.}
We apply existing moderate deviation results (Bahadur--Rao, Cram\'er) to the Bayes risk optimization problem; no new moderate deviation theorem is proved. The analysis rests on three structural assumptions: regular parametric models admitting a Laplace approximation to the marginal likelihood; smooth priors with $\pi(\theta_0)\in(0,\infty)$; and symmetric $0$--$1$ loss. Each of these could be relaxed, potentially changing the threshold law.

Priors with $\pi(\theta_0)=0$ (e.g., priors excluding a neighbourhood of the null) would weaken the type~II risk component and could shift the threshold. Priors with $\pi(\theta_0)=\infty$ introduce logarithmic corrections, as illustrated by the horseshoe example in \S\ref{sec:rs}, where the threshold is modified to $\log n-2\log\log n+O(1)$. Non-symmetric loss functions $L_0\neq L_1$ would change the constants in \eqref{eq:threshold} but not the $\log n$ scaling, provided the loss ratio remains fixed. Singular models (where the Fisher information matrix is degenerate or the parameter space has boundary) may require entirely different asymptotic tools and could produce non-logarithmic threshold laws.

We do not address composite priors with mass away from $\theta_0$ (the ``well-separated'' regime, where large deviation theory governs), nor do we treat sequential testing designs where the sample size is itself a random variable.

\medskip
\noindent\textit{Summary.}
The risk-optimal testing threshold for composite Bayesian testing operates at the moderate deviation scale $\lambda_n\asymp\sqrt{(\log n)/n}$, yielding critical values $t_{\mathrm{crit}}\asymp\sqrt{\log n}$ with
\[
t_{\mathrm{crit}}^2=\log n+O(1).
\]
The leading term $\log n$ is universal across regular parametric models and priors satisfying $\pi(\theta_0)\in(0,\infty)$; the additive constants are determined by the local prior density, Fisher information, and prior model odds. This boundary provides a unified explanation for the Lindley paradox, Jeffreys' testing threshold, the BIC penalty, and the Chernoff--Stein error exponents; phenomena that have appeared separately in the literature but share a common asymptotic origin in the polynomial tail decay at the moderate deviation scale.

\medskip
\noindent\textit{Open directions.}
Several extensions merit investigation. Sequential testing designs, where data arrive over time and the rejection boundary may depend on the stopping rule, present a natural generalization; the connection to e-values and anytime-valid inference \citep{shafer2021testing,grunwald2020safe} suggests that the moderate deviation boundary may have a sequential counterpart. Non-exchangeable settings, where the i.i.d.\ assumption is relaxed, require moderate deviation principles for dependent data; see \citet{dinwoodie1992large} for initial results on exchangeable random variables. The moderate deviation calibration principle extends beyond parametric point-null testing: \citet{PolsonSokolovZantedeschi2026_GOF} show that it governs Bayes-risk optimal thresholds for goodness-of-fit tests, where the same $\sqrt{\log n}$ inflation arises without reliance on Bayes factors or likelihood-ratio structure. The interaction between sparsity and the testing threshold in high-dimensional settings, partially addressed through the horseshoe example in \S\ref{sec:rs}, deserves a more systematic treatment. Finally, the connection to optimal sample size selection \citep{LynchSethuraman1970}, where the question is reversed from ``what threshold given $n$?'' to ``what $n$ given a target evidence level?'', provides a natural decision-theoretic complement to the results developed here. Each of these directions would test, and potentially extend, the central thesis: that moderate deviations provide the natural asymptotic bridge between fixed-threshold testing and fully separated error-exponent theory, and that the $\sqrt{\log n}$ scale is the canonical meeting point.

\appendix

\section{Moderate deviation calculations}\label{app:md}

Cram\'er showed that if the MGF exists in a neighbourhood of the origin:
\[
\frac{1}{n}\log\Pbb(\bar{X}_n>\lambda)\to -I(\lambda),\qquad
I(\lambda)=\sup_{t\in\R}\{t\lambda-\psi(t)\}.
\]
The Bahadur--Rao refinement supplies the polynomial prefactor:
\[
P(|\bar{X}_n|>\lambda_n)\sim
\frac{2\sigma}{\sqrt{2\pi n}\,\lambda_n}\exp\!\left(-\frac{n\lambda_n^2}{2\sigma^2}\right),
\quad\lambda_n\to 0,\ n\lambda_n^2\to\infty.
\]
Substituting $\lambda_n=a\sqrt{(\log n)/n}$:
\[
\Pbb\!\left(|\bar{X}_n|>a\sqrt{\frac{\log n}{n}}\right)\sim
\frac{\sqrt{2}\,\sigma}{a\sqrt{\pi\log n}}\,n^{-a^2/(2\sigma^2)}.
\]

\noindent\textit{Cram\'er's rate function.}
For $\text{Bernoulli}(p_0)$ data, the rate function is the binary KL divergence:
$I(a)=a\log\frac{a}{p_0}+(1-a)\log\frac{1-a}{1-p_0}$.
For $X\sim\text{Bin}(n,p_0)$, the Chernoff bound gives
\[
\alpha=P(X\ge c\mid H_0)\le\inf_{t\ge 0}e^{-tc}\bigl((1-p_0)+p_0\,e^t\bigr)^n,
\]
confirming that the Chernoff and Cram\'er rates coincide for i.i.d.\ data.

\section{Bayes risk computation for the normal model}\label{sec:lindley-risk}

The following derivation follows \citet{SahuSmith2006}. Assume
$X\mid\theta\sim N(\theta,\sigma^2)$ with prior $\pi^{(n)}(\theta)=N(\mu_t,\tau_t^2)$
and scoring prior $\pi^{(s)}(\theta)=N(\mu_s,\tau_s^2)$. The posterior mean and variance are
\[
E(\theta\mid\bar{x}_n)=\lambda_t^2\left(\frac{n\bar{x}_n}{\sigma^2}+\frac{\mu_t}{\tau_t^2}\right),
\qquad
\mathrm{var}(\theta\mid\bar{x}_n)=\lambda_t^2=\frac{1}{n/\sigma^2+1/\tau_t^2}.
\]
Using bivariate normal identities, the risk function reduces to
\[
r(\pi^{(s)},\delta^{\pi^{(n)}}_n)=L_0\,\P\{U>a,V<b\}+L_1\,\P\{U<a,V>b\},
\]
with $(U,V)$ bivariate normal, correlation $\rho=(1+\sigma^2/(n\tau_s^2))^{-1/2}$,
$a=(\theta_0-\mu_s)/\tau_s$, $b=\rho\,\{k^{\pi^{(n)}}(n)-\mu_s\}/\tau_s$.

\medskip\noindent\textit{Lindley's derivation of the $\log n$ threshold.}\label{sec:lindley}
\citet{lindley1957} considers the Wald-test setting. The null hypothesis is accepted if and only if
\begin{equation}
Z^2<A(\phi_0)+\log n,
\label{eq:3.12}
\end{equation}
where $A(\phi_0)$ is finite. By Mills' ratio, the significance level satisfies
\begin{equation}
\alpha_n=2\{1-\Phi(\sqrt{A(\phi_0)+\log n})\}\sim\sqrt{\frac{2}{\pi}}\,e^{-A(\phi_0)/2}\,\frac{1}{\sqrt{n\log n}}.
\label{eq:3.13}
\end{equation}
The decreasing $\alpha_n$ represents the risk-optimal calibration identified in Theorem~\ref{thm:main}.

\section{Dawid's theorem on posterior model probabilities}\label{app:dawid}

Let $M=\{P_\omega:\omega\in\Omega\subset\R^d\}$ be a parametric model with prior $p(\omega)$,
let $Q$ be the true data-generating distribution, and let
$\vartheta(\omega)=p(\omega)/\{\det I(\omega)\}^{1/2}$.

\begin{theorem}[Dawid]\label{thm:dawid}
Under standard regularity:
\begin{enumerate}
\item[\emph{(i)}] If $Q\notin M$, then
$n^{-1}\log\{p_n(X^n)/q_n(X^n)\}\xrightarrow{a.s.}-K(Q,M)$,
where $K(Q,M)=\min_\omega D(Q\|P_\omega)$.
\item[\emph{(ii)}] If $Q=P_{\omega^*}$, then
\begin{multline*}
\log\frac{p_n(X^n)}{q_n(X^n)}=-\frac{d}{2}\log\!\left(\frac{n}{2\pi}\right)\\
+\log\vartheta(\omega^*)+\tfrac{1}{2}\chi^2_d+o_p(1).
\end{multline*}
\end{enumerate}
\end{theorem}

\begin{remark}[Relation to Theorem~\ref{thm:main}]
Theorem~\ref{thm:dawid} is stochastic in nature: it describes the asymptotic behaviour of the log marginal-likelihood ratio under repeated sampling, either through an almost sure limit under misspecification or through a BIC-type expansion under correct specification. This differs from Theorem~\ref{thm:main}, which is deterministic and local in the observed statistic $\bar{x}$, and is uniform over the moderate deviation band relevant for the testing boundary. Accordingly, Theorem~\ref{thm:dawid} supports the same asymptotic picture but does not by itself yield the explicit threshold constant in \eqref{eq:threshold}. It is more general in stochastic scope, but less sharp for the local threshold calculation carried out in Theorem~\ref{thm:main}.
\end{remark}

This result is included to situate the present threshold calculation within the broader asymptotic theory of marginal-likelihood ratios; unlike Theorem~\ref{thm:main}, it is not a uniform expansion in the observed test statistic.

\begin{proof}[Proof sketch]
Decompose $\log(p_n/q_n)=A+B+C$ where
\begin{align*}
A &= \log\frac{p_n(X^n)}{p_n(X^n\mid\hat\omega)},\quad
B = \log\frac{p_n(X^n\mid\hat\omega)}{p_n(X^n\mid\omega^*)},\quad
C = \log\frac{p_n(X^n\mid\omega^*)}{q_n(X^n)}.
\end{align*}
By Laplace expansion \citep{Tierney1986}:
\[
A=-\frac{d}{2}\log\!\left(\frac{n}{2\pi}\right)+\log\frac{p(\omega^*)}{\bigl\{\det I(\omega^*)\bigr\}^{1/2}}+O_p(n^{-1/2}).
\]
By Taylor expansion, $B\xrightarrow{D}\frac{1}{2}\chi^2_d$ (Wilks' theorem). Under $Q$,
$C=-nK(Q,M)+O_p(\sqrt{n})$.

\noindent\textit{Case~(i):} $K(Q,M)>0$ dominates, yielding the almost sure limit.
\textit{Case~(ii):} $K(Q,M)=0$ and $C=0$; combining $A$ and $B$ yields the result \citep{Clarke1990}.
\end{proof}

\bibliographystyle{plainnat}
\bibliography{EJS_revised_testing}

\begin{thebibliography}{26}
\providecommand{\natexlab}[1]{#1}
\providecommand{\url}[1]{\texttt{#1}}
\expandafter\ifx\csname urlstyle\endcsname\relax
  \providecommand{\doi}[1]{doi: #1}\else
  \providecommand{\doi}{doi: \begingroup \urlstyle{rm}\Url}\fi

\bibitem[Bahadur and Rao(1960)]{bahadur1960deviations}
Raghu~Raj Bahadur and R.~Ranga Rao.
\newblock On deviations of the sample mean.
\newblock \emph{Annals of Mathematical Statistics}, 31\penalty0 (4):\penalty0
  1015--1027, 1960.

\bibitem[Bartlett(1957)]{bartlett1957comment}
M.~S. Bartlett.
\newblock {A comment on {D}.~{V}.~{Lindley}'s statistical paradox}.
\newblock \emph{Biometrika}, 44\penalty0 (3/4):\penalty0 533--534, 1957.

\bibitem[Berger(2003)]{berger2003could}
James~O. Berger.
\newblock Could {F}isher, {J}effreys and {N}eyman have agreed on testing?
\newblock \emph{Statistical Science}, 18\penalty0 (1):\penalty0 1--32, 2003.
\newblock \doi{10.1214/ss/1056397485}.

\bibitem[Carvalho et~al.(2010)Carvalho, Polson, and
  Scott]{carvalho2010horseshoe}
Carlos~M. Carvalho, Nicholas~G. Polson, and James~G. Scott.
\newblock The horseshoe estimator for sparse signals.
\newblock \emph{Biometrika}, 97\penalty0 (2):\penalty0 465--480, 2010.

\bibitem[Chernoff(1952)]{Chernoff1952_MeasureAsymptoticEfficiency}
Herman Chernoff.
\newblock A measure of asymptotic efficiency for tests of a hypothesis based on
  the sum of observations.
\newblock \emph{Annals of Mathematical Statistics}, 23\penalty0 (4):\penalty0
  493--507, 1952.
\newblock \doi{10.1214/aoms/1177729330}.

\bibitem[Clarke and Barron(1990)]{Clarke1990}
Bertrand~S. Clarke and Andrew~R. Barron.
\newblock Information-theoretic asymptotics of {B}ayes methods.
\newblock \emph{{IEEE} Transactions on Information Theory}, 36\penalty0
  (3):\penalty0 453--471, 1990.

\bibitem[Daniels(1954)]{daniels1954saddlepoint}
Henry~E. Daniels.
\newblock Saddlepoint approximations in statistics.
\newblock \emph{Annals of Mathematical Statistics}, 25\penalty0 (4):\penalty0
  631--650, 1954.

\bibitem[Datta and Ghosh(2013)]{datta2013asymptotic}
Jyotishka Datta and Jayanta~K. Ghosh.
\newblock Asymptotic properties of {B}ayes risk for the horseshoe prior.
\newblock \emph{Bayesian Analysis}, 8\penalty0 (1):\penalty0 111--132, 2013.
\newblock \doi{10.1214/13-BA805}.

\bibitem[Dawid(1984)]{dawid1984statistical}
A.~P. Dawid.
\newblock {Present position and potential developments: Some personal views:
  {S}tatistical theory: The prequential approach}.
\newblock \emph{Journal of the Royal Statistical Society: Series A},
  147\penalty0 (2):\penalty0 278--292, 1984.

\bibitem[Dinwoodie and Zabell(1992)]{dinwoodie1992large}
Ian~H. Dinwoodie and Sandy~L. Zabell.
\newblock Large deviations for exchangeable random vectors.
\newblock \emph{Annals of Probability}, 20\penalty0 (3):\penalty0 1147--1166,
  1992.
\newblock \doi{10.1214/aop/1176989683}.

\bibitem[Efron and Truax(1968)]{EfronTruax1968}
Bradley Efron and Donald Truax.
\newblock Large deviations theory in exponential families.
\newblock \emph{Annals of Mathematical Statistics}, 39\penalty0 (5):\penalty0
  1402--1424, 1968.

\bibitem[Eichelsbacher and Ganesh(2002)]{eichelsbacher2002moderate}
Peter Eichelsbacher and Ayalvadi~J. Ganesh.
\newblock Moderate deviations for {B}ayes posteriors.
\newblock \emph{Scandinavian Journal of Statistics}, 29\penalty0 (1):\penalty0
  153--167, 2002.

\bibitem[Gr{\"u}nwald et~al.(2024)Gr{\"u}nwald, de~Heide, and
  Koolen]{grunwald2020safe}
Peter Gr{\"u}nwald, Rianne de~Heide, and Wouter~M. Koolen.
\newblock Safe testing.
\newblock \emph{Journal of the Royal Statistical Society: Series B (Statistical
  Methodology)}, 86\penalty0 (5):\penalty0 1091--1128, 2024.
\newblock \doi{10.1093/jrsssb/qkae011}.

\bibitem[Jeffreys(1935)]{Jeffreys1935}
Harold Jeffreys.
\newblock Some tests of significance, treated by the theory of probability.
\newblock \emph{Proceedings of the Cambridge Philosophical Society},
  31:\penalty0 203--222, 1935.

\bibitem[Jeffreys(1961)]{jeffreys1961}
Harold Jeffreys.
\newblock \emph{Theory of Probability}.
\newblock Oxford University Press, Oxford, 3rd edition, 1961.

\bibitem[Lindley(1957)]{lindley1957}
D.~V. Lindley.
\newblock A statistical paradox.
\newblock \emph{Biometrika}, 44\penalty0 (1/2):\penalty0 187--192, 1957.

\bibitem[Lopes and Polson(2019)]{LopesPolson2019_Redux}
Hedibert~F. Lopes and Nicholas~G. Polson.
\newblock Bayesian hypothesis testing: Redux.
\newblock \emph{Brazilian Journal of Probability and Statistics}, 33\penalty0
  (4):\penalty0 745--755, 2019.
\newblock \doi{10.1214/19-BJPS442}.

\bibitem[Lynch and Sethuraman(1970)]{LynchSethuraman1970}
James Lynch and Jayaram Sethuraman.
\newblock A comparison of two sampling schemes when testing a simple hypothesis
  versus a simple alternative.
\newblock \emph{Sankhy\=a: The Indian Journal of Statistics, Series~A},
  32:\penalty0 299--310, 1970.

\bibitem[Polson et~al.(2026{\natexlab{a}})Polson, Sokolov, and
  Zantedeschi]{PolsonSokolovZantedeschi2026_Evalues}
Nicholas~G. Polson, Vadim Sokolov, and Daniel Zantedeschi.
\newblock Bayes, {E}-values and testing.
\newblock \emph{arXiv preprint arXiv:2602.04146}, 2026{\natexlab{a}}.
\newblock arXiv:2602.04146.

\bibitem[Polson et~al.(2026{\natexlab{b}})Polson, Sokolov, and
  Zantedeschi]{PolsonSokolovZantedeschi2026_GOF}
Nicholas~G. Polson, Vadim Sokolov, and Daniel Zantedeschi.
\newblock Bayes risk for goodness of fit tests.
\newblock \emph{arXiv preprint arXiv:2602.15297}, 2026{\natexlab{b}}.
\newblock arXiv:2602.15297.

\bibitem[Rubin and Sethuraman(1965{\natexlab{a}})]{RubinSethuraman1965_BRE}
Herman Rubin and Jayaram Sethuraman.
\newblock Bayes risk efficiency.
\newblock \emph{Sankhy\=a: The Indian Journal of Statistics, Series~A},
  27\penalty0 (3):\penalty0 347--356, 1965{\natexlab{a}}.

\bibitem[Rubin and Sethuraman(1965{\natexlab{b}})]{rubin1965probabilities}
Herman Rubin and Jayaram Sethuraman.
\newblock Probabilities of moderate deviations.
\newblock \emph{Sankhy\=a: The Indian Journal of Statistics, Series~A},
  27\penalty0 (3):\penalty0 325--346, 1965{\natexlab{b}}.

\bibitem[Sahu and Smith(2006)]{SahuSmith2006}
Sujit~K. Sahu and T.~M.~F. Smith.
\newblock A {B}ayesian method of sample size determination with practical
  applications.
\newblock \emph{Journal of the Royal Statistical Society: Series A (Statistics
  in Society)}, 169\penalty0 (2):\penalty0 235--253, 2006.
\newblock \doi{10.1111/j.1467-985X.2006.00408.x}.

\bibitem[Sellke et~al.(2001)Sellke, Bayarri, and Berger]{sellke2001calibration}
Thomas Sellke, Mar{\'\i}a~Jes{\'u}s Bayarri, and James~O. Berger.
\newblock Calibration of $p$ values for testing precise null hypotheses.
\newblock \emph{The American Statistician}, 55\penalty0 (1):\penalty0 62--71,
  2001.

\bibitem[Shafer(2021)]{shafer2021testing}
Glenn Shafer.
\newblock Testing by betting: A strategy for statistical and scientific
  communication.
\newblock \emph{Journal of the Royal Statistical Society Series A: Statistics
  in Society}, 184\penalty0 (2):\penalty0 407--431, 2021.
\newblock \doi{10.1111/rssa.12647}.

\bibitem[Tierney and Kadane(1986)]{Tierney1986}
Luke Tierney and Joseph~B. Kadane.
\newblock Accurate approximations for posterior moments and marginal densities.
\newblock \emph{Journal of the American Statistical Association}, 81\penalty0
  (393):\penalty0 82--86, 1986.

\end{thebibliography}

\end{document}